\documentclass[11pt,reqno]{amsart}
\usepackage{amsmath,amstext,amssymb,amscd}
\usepackage{verbatim}
\usepackage{enumerate}
\usepackage{mathrsfs}
\usepackage[dvipsnames,usenames]{xcolor}

\usepackage{amsthm}

\newtheorem{theorem}{Theorem}[section]
\newtheorem{lemma}[theorem]{Lemma}

\newtheorem{proposition}[theorem]{Proposition}

\theoremstyle{definition}

\theoremstyle{remark}

\numberwithin{equation}{section}

\def\ee{\epsilon}

\def\om{\omega}

\def\la{\lambda}

\def\om{\omega}

\begin{document}
\title[2D NSE with horizontal viscosity]
{On the two-dimensional Navier-Stokes equations with horizontal viscosity}

\date{September 15, 2024}

\author[C. Cao]{Chongsheng Cao}
\address{Department of Mathematics \& Statistics \\Florida International University\\
Miami, Florida 33199, USA} \email{caoc@fiu.edu}

\author[Y. Guo]{Yanqiu Guo}
\address{Department of Mathematics \& Statistics \\Florida International University\\
Miami, Florida 33199, USA} \email{yanguo@fiu.edu}

\begin{abstract}
This paper is concerned with a 2D channel flow that is periodic horizontally but bounded above and below by hard walls. We assume the presence of horizontal viscosity only. We study the well-posedness, large-time behavior, and stability of solutions. 
For global well-posedness, we aim to assume less differentiability on initial velocity $(u_0, v_0)$: in particular, we assume $u_0,v_0\in L^2(\Omega)$ and $\partial_y u_0 \in L^2(\Omega)$.
\end{abstract}

\maketitle

\section{Introduction}

\subsection{Motivation}
Navier-Stokes equations (NSEs) are fundamental equations for viscous fluid flows, whereas Euler equations describe inviscid fluids. In two dimensions, NSEs and Euler equations are both globally well-posedness. 
In particular, 2D NSEs have a unique global weak solution if the initial velocity belongs to $L^2(\Omega)$, while 2D Euler equations have a unique global weak solution if the initial vorticity is taken from $L^{\infty}(\Omega)$ \cite{Yudo}. Mathematically, it is interesting to study 2D NSEs with only horizontal viscosity, and ask the regularity of initial data needed to establish the global well-posedness of weak solutions. We aim to assume less differentiability on initial velocity $(u_0, v_0)$: in particular, we assume 
$u_0,v_0\in L^2(\Omega)$ with $\partial_y u_0 \in L^2(\Omega)$ to obtain the existence and uniqueness of global weak solutions. Our equation can be considered as a model ``between" NSEs and Euler equations in 2D. Although this paper focuses on a 2D fluid flow, it is worth mentioning that, solutions of 3D Euler equations may form singularity in finite time \cite{Elgindi}, while the global regularity of 3D NSEs and the uniqueness of Leray-Hopf weak solutions remain major open problems. However, with special forcing terms, non-uniqueness of Leray solutions to 3D forced NSEs was shown in \cite{ABC}, and finite time blow-up for the hypodissipative NSEs was studied in \cite{CMZ}.

Consider 2D Navier-Stokes equations with only horizontal viscosity:
\begin{align}  \label{PDE}
\mathbf u_t - \mathbf u_{xx} + (\mathbf u\cdot \nabla)\mathbf u +  \nabla p = \mathbf f(x,y,t), 
\end{align}
with $\nabla\cdot  \mathbf u = u_x + v_y=0$, where $\mathbf u = (u,v)$ represents the velocity field of the 2D fluid. 
We consider a 2D flow in the channel $(x,y) \in \mathbb R \times [0,1]$. Assume that the flow satisfies the boundary condition 
\begin{align} \label{boundary}
v|_{y=0}=v|_{y=1}=0. 
\end{align}
Moreover, we assume that velocity $\mathbf u=(u,v)$, pressure $p$ and force $\mathbf f=(f_1, f_2)$ are periodic in the $x$ variable.
Without loss of generality, we assume that the force $\mathbf f$ is divergence free and satisfies that $f_2|_{y=0}=f_2|_{y=1}=0$. In the paper, we consider a typical period $(x,y)\in \Omega = \mathbb T \times [0,1]$. Here, $\mathbb T$ is the 1D torus (i.e., a circle) identified with the interval $[0,1]$. Thus, geometrically, $\Omega = \mathbb T \times [0,1]$ is identified with a cylindrical surface of height 1. Physically, this setting represents a 2D channel flow which is periodic in the $x$-direction, 
and the boundaries $y=0$ and $y=1$ are impenetrable, so the normal component of the velocity field vanishes at the boundaries.

In paper \cite{LZZ}, equation (\ref{PDE}) was considered on $\mathbb R^2$ (or on $\mathbb T^2$) without forcing, and global well-posedness of weak solutions was shown by assuming initial velocity $(u_0,v_0)$ satisfies that 
$\partial_y u_0$ and $\partial_y v_0$ both in $L^2(\mathbb R^2)$. A stochastic version of (\ref{PDE}) was also considered in \cite{LZZ}. The novelty of our well-posedness result lies in that we assume less differentiability on initial data; in particular, we assume only $\partial_y u_0$ belongs to $L^2(\Omega)$. Please see Theorem \ref{thm-1}. 

Let us remark that the 2D Euler equations for inviscid fluid are globally well-posed and the upper bound of the $H^s$ norm of the velocity grows with a double exponential rate as $t\rightarrow \infty$, if the initial velocity belongs to $H^s(\Omega)$ with $s>2$.
When the domain is a disk with boundary, an initial data for the 2D Euler equations was constructed in \cite{KS} for which the gradient of vorticity exhibits double exponential growth in time.
One of our results for equation (\ref{PDE}) shows that if $\mathbf u_0 \in H^2(\Omega)$, then the $H^2$ norm of $\mathbf u$ is bounded by a constant for all time, if we impose a restriction on the forcing $\mathbf f$. 
Please see Theorem \ref{thm-3}. A similar result for equation (\ref{PDE}) on $\mathbb T \times \mathbb R$ without forcing was shown in \cite{DWXZ}. 

Finally, we would like to mention that, the non-uniqueness of 2D Euler equations was established in \cite{V1,V2} by constructing initial vorticity and external force in some $L^p(\mathbb R^2)$ space. 
Also, non-uniqueness of Leray solutions for 2D forced NSEs with hypodissipation was shown in \cite{AC}.

\smallskip

\subsection{Main results}
Let $\Omega = \mathbb T \times [0,1]$, where $\mathbb T$ is the 1D torus identified with $[0,1]$.

We define the function space
\begin{align}  \label{space}
\mathcal H =  \Big\{ &\mathbf u = (u,v) \in (L^2(\Omega))^2:     u=\sum_{\mathbf k \in \mathbb Z^2} a_{\mathbf k} e^{2 i \pi \mathbf k \cdot \mathbf x}, \; \overline{a_{\mathbf k}} = a_{-\mathbf k}, \notag\\
&v=\sum_{k_1 \in \mathbb Z, \, k_2\in \mathbb Z^+} b_{(k_1,k_2)} e^{2i \pi k_1 x} \sin (\pi k_2 y),    \; \overline{b_{(k_1, k_2)}} = b_{(-k_1, k_2)}  ,   \;\;\text{and}  \; \nabla \cdot \mathbf u =0     \Big\}.
\end{align}
The function space $\mathcal H$ consists of all $L^2(\Omega)$ real-valued divergence-free 2D vector fields $(u,v)$ such that $v|_{y=0}=v|_{y=1} =0$.

Denote the space 
\begin{align} \label{def-V}
\mathcal V = \mathcal H \cap (H^1(\Omega))^2. 
\end{align}

Throughout, we denote the $L^p(\Omega)$ norm by $\|\cdot\|_p$, $1\leq p \leq \infty$.

First, we state our major result concerning the global well-posedness of weak solutions to equation (\ref{PDE}).

\begin{theorem} [Global well-posedness]  \label{thm-1}
Assume initial velocity $\mathbf u_0 = (u_0,v_0) \in \mathcal H$ with $\partial_y u_0 \in L^2(\Omega)$.  
Also, assume that $\mathbf f = (f_1,f_2)  \in L^1(0,T; \mathcal H)$ with $\partial_y f_1\in L^2(0,T; L^2(\Omega))$ for any $T>0$. 
Then equation (\ref{PDE}) has a unique global weak solution satisfying that, for any $T>0$, 
$\mathbf u \in C([0,T];\mathcal H)$, $u_y \in L^{\infty}(0,T;L^2(\Omega))$, 
$u_x, v_x, u_{xy}\in L^2(0,T;L^2(\Omega))$, and
\begin{align} \label{weak-f}
\frac{d}{dt} (\mathbf u, \mathbf w)_{L^2} + (\mathbf u_x, \mathbf w_x)_{L^2} + \int_{\Omega}  [(\mathbf u\cdot \nabla)\mathbf u] \cdot \mathbf w dx dy = (\mathbf f, \mathbf w)_{L^2},    \;\; t>0,
\end{align}
for any $\mathbf w \in \mathcal V$.
\end{theorem}

\smallskip

The next theorem is about large time behavior of solutions. In particular, it states that, as $t\rightarrow \infty$, the $H^2$ norm of the velocity is bounded by a constant dependent on the initial velocity, if we impose a restriction on the forcing term. 

\begin{theorem} [Uniform bound of the $H^2$-norm]     \label{thm-3}
Assume initial velocity $ \mathbf u_0 \in    \mathcal H \cap  H^2(\Omega)$. Moreover, we assume that $\int_0^{\infty} \Big(  \|\mathbf f\|_2    +    \|\emph{curl} \, \mathbf f\|_2 +\|\emph{curl} \, \mathbf f\|_2^2 +  \|\partial_{yy} f_1\|_2^2      \Big) dt \leq C_0$ and $\overline{f_1}=0$. Then equation (\ref{PDE}) has a global strong solution satisfying
\begin{align} \label{H2b}
\|\mathbf u(t)\|_{H^2(\Omega)} \leq  M( \|\mathbf u_0\|_{H^2(\Omega)}, C_0),   \;\;\text{for all}  \;\; t\geq 0,
\end{align}
where $M( \|\mathbf u_0\|_{H^2(\Omega)}, C_0)$ is a constant depending only on $\|\mathbf u_0\|_{H^2(\Omega)}$ and $C_0$.
\end{theorem}

\smallskip

For any function $g\in L^1(\Omega)$, we denote by
\begin{align} \label{ave}
\bar{g}(y) = \int_0^1 g(x,y) dx,  \qquad \tilde{g}(x,y) = g(x,y)- \bar{g}(y).
\end{align}
The next result is concerned with the stability around the shear flows.

\begin{proposition}[Stability around shear flows $(ay, 0)$]  \label{stability}
Assume $\emph{curl} \, \mathbf f =0$ and initial velocity $ \mathbf u_0 \in    \mathcal H \cap  H^2(\Omega)$. Let $(u+ay, v)$ be a solution of equation (\ref{PDE}) and we denote by 
$\om^*= v_x-u_y$. Then, the oscillation part of $\om^*$ decays exponentially to zero:
\begin{align*}
\|\widetilde{\om^*}(t)\|_2^2  \leq
 e^{- \alpha t}    \|\widetilde{\om^*} (0)\|_2^2,    \;\;   \text{for}  \;\;  t\geq 0,
\end{align*}
where $\alpha>0$, if $\|\omega_0\|_2$ and $\|\nabla \omega_0\|_2$ are sufficiently small 
so that (\ref{om-9}) holds.

\end{proposition}

\smallskip

The last result is about the asymptotic behavior of solutions if there is no forcing term. 

\begin{proposition}[Asymptotic solutions]\label{INFTY}
Assume $\mathbf f=0$ and $\mathbf u_0 \in \mathcal H \cap (H^1(\Omega))^2$. Let $\mathbf u= (u, v)$ be the global solutions of (\ref{PDE}). Then,
\begin{eqnarray*}
&&
 \lim_{t\rightarrow \infty} (\tilde{u}, v) = (0,0)  \;\;\text{in} \;\; \mathcal H,  \\
&& \lim_{t\rightarrow \infty} \bar{u} = \phi    \;\;\text{in} \;\;  L^2(0,1),
\end{eqnarray*}
where $(\phi(y),0)$ is a steady state solution of equation (\ref{PDE}).
\end{proposition}

\bigskip

\section{Preliminaries}

In this section, we provide some useful inequalities and identities for our analysis.

\begin{proposition}
Assume that $f$, $g$, $g_x$, $h$ and $h_y\in L^2(\Omega)$. Then the following inequality holds
\begin{align} \label{ineq}
\int_{\Omega} |f g h| \, dx dy \leq C \|f\|_2       \|g\|_2^{1/2}  \left(\|g\|_2^{1/2} + \|g_x\|_2^{1/2}\right) \|h\|_2^{1/2}   \left(\|h\|_2^{1/2} +  \|h_y\|_2^{1/2}\right).
\end{align}
\end{proposition}

\begin{proof}
Using H\"older's inequality and Agmon's inequality, we deduce
\begin{align*}
&\int_{\Omega} |f g h| dx dy \\
& \leq \int_0^1  \Big(\int_0^1 f^2 dx\Big)^{1/2}  \Big(\int_0^1 h^2 dx\Big)^{1/2}  \Big(\sup_{x\in [0,1]} |g| \Big) dy \\
&\leq  C\int_0^1  \Big(\int_0^1 f^2 dx\Big)^{1/2}  \Big(\int_0^1 h^2 dx\Big)^{1/2}  \Big(\int_0^1 g^2 dx\Big)^{1/4}  \Big(\int_0^1 (g^2 + g_x^2) dx\Big)^{1/4} dy \\
&\leq C\|f\|_2 \|g\|_2^{1/2}   \left(\|g\|_2^{1/2} + \|g_x\|_2^{1/2} \right) \sup_{y\in [0,1]} \Big( \int_0^1 h^2 dx \Big)^{1/2} \\
&\leq C\|f\|_2 \|g\|_2^{1/2}       \left(  \|g\|_2^{1/2} + \|g_x\|_2^{1/2}   \right)     \Big(\int_0^1     \big(\int_0^1 h^2 dy\big)^{1/2}  \big(\int_0^1 (h^2 + h_y^2) dy\big)^{1/2}   dx \Big)^{1/2}\\
&\leq C  \|f\|_2 \|g\|_2^{1/2}   \left( \|g\|_2^{1/2} + \|g_x\|_2^{1/2} \right) \|h\|_2^{1/2} \left(\|h\|_2^{1/2} + \|h_y\|_2^{1/2}\right).
\end{align*}
\end{proof}

\smallskip

\begin{proposition}
Assume that $f$, $f_x$, $f_y$ and $f_{xy}\in L^1(\Omega)$. Then the following inequality holds
\begin{align} \label{ineq2}
\|f\|_{\infty} \leq \|f\|_1 + \|f_x\|_1 + \|f_y\|_1 + \|f_{xy}\|_1.
\end{align}
\end{proposition}

\begin{proof}
For any continuous function $\phi(s)$ on $[0,1]$, there exists $s_0\in [0,1]$ such that $\phi(s_0) = \int_0^1 \phi(s) ds$.
Hence, for every $s\in [0,1]$, one has $|\phi(s)| = \left|\int_{s_0}^s \phi'(\xi) d\xi  + \phi(s_0) \right| \leq \int_0^1 |\phi'(\xi)| d\xi + \int_0^1 |\phi(\xi)| d\xi$ if $\phi$ is $C^1$.
Therefore, for any $C^2$ function $f(x,y)$ on $\Omega$, and for all $(x,y) \in \Omega$, we have
\begin{align}
&|f(x,y)|  \leq  \int_0^1 \left|f_x(x, y)\right| dx + \int_0^1 |f(x,y)| dx  \notag\\
& \leq  \int_0^1 \Big( \int_0^1 |f_{xy}(x,y)| dy + \int_0^1 |f_{x}(x,y)| dy \Big) dx + \int_0^1 \Big( \int_0^1 |f_{y}(x,y)| dy + \int_0^1 |f(x,y)| dy \Big) dx \notag\\
& \leq  \|f_{xy}\|_1 +  \|f_{x}\|_1 +  \|f_{y}\|_1 + \|f\|_1.
\end{align}
Thus, $\|f\|_{\infty} \leq \|f_{xy}\|_1 +  \|f_{x}\|_1 +  \|f_{y}\|_1 + \|f\|_1$ if $f\in C^2(\Omega)$. 
Finally, (\ref{ineq2}) can be obtained by a standard density argument using mollifiers.
\end{proof}

\smallskip

Assume that vector fields $\mathbf u(x,y) =(u(x,y),v(x,y))$ and $\mathbf w(x,y)$ are periodic in $x \in \mathbb T$. Also, assume that $v|_{y=0}=v|_{y=1}=0$. 
Moreover, assume that $\nabla \cdot \mathbf u = u_x + v_y =0$.
Then, for sufficiently regular $\mathbf u$ and $\mathbf w$, it is easy to verify that
\begin{align} \label{vanish}
\int_{\Omega} [(\mathbf u \cdot \nabla) \mathbf w] \cdot \mathbf w dx dy =0.
\end{align}

Assume $v,v_y \in L^2(\Omega)$ and $v|_{y=0}=0$. Then we have the following Poinc\'are-type inequality:
\begin{align}  \label{Poin}
\|v\|_2^2 \leq \|v_y\|_2^2.
\end{align}
In fact, since $v|_{y=0}=0$, then
\begin{align*}
&\|v\|_2^2 = \int_{\Omega} v^2 dy dx =   \int_{\Omega} \Big(\int_0^y v_{\xi}(x,\xi) d\xi \Big)^2 dy dx
\leq \int_0^1 \Big(\int_0^1 |v_{\xi}(x,\xi)| d\xi \Big)^2 dx \leq  \|v_y\|_2^2.
\end{align*}

\bigskip

\section{Global well-posedness of weak solutions}
Assume initial data $\mathbf u_0 =(u_0,v_0) \in \mathcal H$ and $\partial_y u_0 \in L^2(\Omega)$. We show the global well-posedness of weak solutions, stated in Theorem \ref{thm-1}.

\subsection{Global existence of weak solutions}

We shall first conduct \emph{a priori} estimates, and then use them to justify the global existence of weak solutions rigorously in Subsection \ref{rigor} via an approximation scheme. 

\subsubsection{Estimate for $\|\mathbf u\|_2$}  \label{L2est}
Take the inner product of equation (\ref{PDE}) with $\mathbf u$. Using equality (\ref{vanish}), we obtain
\begin{align}  \label{e-0}
\frac{1}{2}\frac{d}{dt}\|\mathbf u\|_2^2 +  \|\mathbf u_x\|_2^2 =  \int_{\Omega} \mathbf f \cdot  \mathbf u dx dy \leq \|\mathbf f\|_2 \|\mathbf u\|_2.
\end{align}
Thus, $\|\mathbf u\|_2 \Big(\frac{d}{dt} \|\mathbf u\|_2 \Big) \leq \|\mathbf f\|_2 \|\mathbf u\|_2$. Therefore, 
$\frac{d}{dt} \|\mathbf u\|_2 \leq  \|\mathbf f\|_2$. It follows that
\begin{align}  \label{e-1}
\|\mathbf u(t)\|_2 \leq \|\mathbf u_0\|_2 + \int_0^t  \|\mathbf f\|_2 ds, \;\;\; t\geq 0.
\end{align}
By integrating (\ref{e-0}) over $[0,t]$ and using (\ref{e-1}), we obtain, for all $t\geq 0$,
\begin{align}  \label{e-2}
&\int_0^t  \|\mathbf u_x\|_2^2 ds \leq \frac{1}{2} \|\mathbf u_0\|_2^2 +  \int_0^t  \|\mathbf f\|_2 \|\mathbf u\|_2 ds  \notag\\
&\leq  \frac{1}{2} \|\mathbf u_0\|_2^2 +  \Big(\int_0^t  \|\mathbf f\|_2 ds\Big)  \Big(\|\mathbf u_0\|_2 + \int_0^t  \|\mathbf f\|_2 ds\Big)
\leq \|\mathbf u_0\|_2^2 + 2 \Big(\int_0^t  \|\mathbf f\|_2 ds\Big)^2.
\end{align}

\smallskip

\subsubsection{Estimate for $\|u_y\|_2$} \label{uyest}
Equation (\ref{PDE}) can be written as a system of two coupled equations:
\begin{align} 
& u_t - u_{xx} + u u_x + v u_y + p_x = f_1   \label{be-3}   \\
& v_t - v_{xx} + u v_x + v v_y + p_y = f_2.       \label{be-3'}
\end{align}
Differentiate (\ref{be-3}) with respect to $y$, and take the inner product with $u_y$. Then
\begin{align} \label{be-4}
&\frac{1}{2} \frac{d}{dt} \|u_{y}\|_2^2 + \|u_{xy}\|_2^2 = - \int_{\Omega} p_{xy} u_y dx dy +   \int_{\Omega} (\partial_y f_1) u_y dx dy \notag\\
&= \int_{\Omega} p_y u_{xy} dx dy +   \int_{\Omega} (\partial_y f_1) u_y dx dy.
\end{align}

Differentiate (\ref{be-3}) and (\ref{be-3'}) by $x$ and $y$ respectively, and add the results. Then, we use the divergence free condition that $u_x + v_y=0$ and $\partial_x f_1 + \partial_y f_2 =0$
to get
\begin{align} \label{Delta-p}
-\Delta p = (u^2)_{xx} + (v^2)_{yy} + 2 (uv)_{xy}.
\end{align}
Moreover, using (\ref{be-3'}) as well as the boundary condition that $v|_{y=0} = v|_{y=1} =0$ and the assumption that $f_2|_{y=0} =f_2|_{y=1}=0$, we get the boundary condition for the pressure $p$:
\begin{align} \label{bd-p}
p_y|_{y=0}=p_y|_{y=1}=0.
\end{align}

Let $\phi$ be the solution of the Poisson equation:
\begin{align}
&-\Delta \phi = u_y, \label{PHI-1} \\
&\phi|_{y=0}=\phi|_{y=1}=0. \label{PHI-2}
\end{align}

Using integration by parts, we have
\begin{align*}
&\int_{\Omega} (\Delta p_y )  \phi_{x} dx dy = \int_{\Omega} ( p_{xxy} + p_{yyy})  \phi_{x} dx dy  \\
& =  \int_{\Omega} p_y  (\Delta \phi_{x})  dx dy   + \int_{\Omega}  \left[p_{yy}  \phi_{x} \right]_{y=0}^{y=1} dx  -
\int_{\Omega} \left[ p_{y} \phi_{xy} \right]_{y=0}^{y=1} dx   \\
& =  \int_{\Omega} p_y  (\Delta \phi_{x})  dx dy = -\int_{\Omega} p_y u_{xy} dx dy,
\end{align*}
where we have used the boundary values (\ref{PHI-2}) and (\ref{bd-p}). Therefore,
\begin{align}\label{main}
&\int_{\Omega} p_y u_{xy} dx dy =  \int_{\Omega} (-\Delta p_y )  \phi_{x} dx dy    =  \int_{\Omega} [ (u^2)_{xxy} + (v^2)_{yyy} + 2 (uv)_{xyy}] \phi_{x} dx dy,
\end{align}
where the last equality is due to (\ref{Delta-p}).

Because of (\ref{be-4}) and (\ref{main}), we have
\begin{align}  \label{MAIN}
\frac{1}{2} \frac{d}{dt} \|u_{y}\|_2^2 + \|u_{xy}\|_2^2 = \int_{\Omega} [ (u^2)_{xxy} + (v^2)_{yyy} + 2 (uv)_{xyy}] \phi_{x} dx dy +   \int_{\Omega} (\partial_y f_1) u_y dx dy.
\end{align}

Before evaluating (\ref{MAIN}), let us have some estimates on $\phi$. Take $L^2-$inner product of (\ref{PHI-1}) with $\phi$ and use
boundary condition (\ref{PHI-2}) to get

\begin{align*}
\|\phi_{x}\|_2^{2}   +   \|\phi_{y}\|_2^2 =  \int_{\Omega} u_y \phi dx dy  = - \int_{\Omega} u \phi_y dx dy \leq \|u \|_2\|\phi_{y}\|_2.
\end{align*}
By the Cauchy--Schwarz inequality we get
\begin{align}  \label{PHI-L2}
& \|\phi_{x}\|_2^{2}   +   \|\phi_{y}\|_2^2 \leq \|u \|_2^2.
\end{align}

Take $L^2-$norm of both sides of (\ref{PHI-1}) and use boundary condition (\ref{PHI-2}) to get
\begin{align} \label{PHI-H2}
& \|\phi_{xx}\|_2^{2}   +   \|\phi_{yy}\|_2^2 + 2  \int_{\Omega} \phi_{xx} \phi_{yy} dx dy
= \|\phi_{xx}\|_2^{2}   +   \|\phi_{yy}\|_2^2 + 2  \|\phi_{xy}\|_2^2 =  \|u_y \|_2^2.
\end{align}

Also, take $L^2-$inner product of (\ref{PHI-1}) with $\phi_{xxxx}$  and use
boundary condition (\ref{PHI-2}) to get
\begin{align*} 
\|\phi_{xxx}\|_2^{2}   +   \|\phi_{xxy}\|_2^2 =   \int_{\Omega} u_y \phi_{xxxx} dx dy  = - \int_{\Omega} u_{xy} \phi_{xxx} dx dy \leq \|u_{xy} \|_2\|\phi_{xxx}\|_2.
\end{align*}
Using the Cauchy--Schwarz inequality gives
\begin{align}  \label{PHI-H3}
\| \phi_{xxx} \|_2^2   +   \| \phi_{xxy} \|_2^2  \leq  \| u_{xy} \|_2^2.
\end{align}

Now we deal with the right-hand side of (\ref{MAIN}) term by term. First we look at
\begin{align}     \label{lo-1}
 \int_{\Omega}  (u^2)_{xxy}  \phi_{x} dx dy = - 2 \int_{\Omega} u u_{xy} \phi_{xx} dx dy
 -2 \int_{\Omega} u_x u_y \phi_{xx} dx dy.
\end{align}
Applying inequality (\ref{ineq}) and estimate (\ref{PHI-H2}) \& (\ref{PHI-H3}), we evaluate
\begin{align}  \label{be-5}
&\left| \int_{\Omega} u u_{xy} \phi_{xx} dx dy  \right| \notag\\
&\leq C  \|u\|_2^{1/2}  \left( \|u\|_2^{1/2} + \|u_x\|_2^{1/2} \right) \|u_{xy}\|_2  \|\phi_{xx}\|_2^{1/2} \left( \|\phi_{xx}\|_2^{1/2}   +   \|\phi_{xxy}\|_2^{1/2} \right) \notag\\
&\leq C  \|u\|_2^{1/2}  \left( \|u\|_2^{1/2} + \|u_x\|_2^{1/2} \right) \|u_{xy}\|_2  \|u_y\|_2^{1/2} \left( \|u_y\|_2^{1/2}   +   \|u_{xy}\|_2^{1/2} \right) \notag\\
& =  C  \|u\|_2^{1/2}   \|u_x\|_2^{1/2}      \|u_{xy}\|_2^{3/2}   \|u_y\|_2^{1/2}  +       C  \|u\|_2      \|u_{xy}\|_2^{3/2}   \|u_y\|_2^{1/2}     \notag\\
& \hspace{0.2 in} + C  \|u\|_2^{1/2}   \|u_x\|_2^{1/2}      \|u_{xy}\|_2  \|u_y\|_2   + C    \|u\|_2     \|u_{xy}\|_2  \|u_y\|_2  \notag\\
&\leq  \epsilon \|u_{xy}\|_2^2 + C\|u\|_2^2 \|u_x\|_2^2  \|u_y\|_2^2 + C\|u\|_2^4 \|u_y\|_2^2 + C\|u\|_2^2  \|u_y\|_2^2 + C\|u_x\|_2^2  \|u_y\|_2^2.
\end{align}
Also,
\begin{align} \label{be-6}
&\left|\int_{\Omega} u_x u_y \phi_{xx} dx dy \right| \notag\\
&\leq C \|u_x\|_2  \|u_y\|_2^{1/2}  \left( \|u_y\|_2^{1/2} +  \|u_{xy}\|_2^{1/2} \right)   \|\phi_{xx}\|_2^{1/2} \left (  \|\phi_{xx}\|_2^{1/2} +  \|\phi_{xxy}\|_2^{1/2} \right)  \notag\\
&\leq C \|u_x\|_2  \|u_y\|_2^{1/2}  \left( \|u_y\|_2^{1/2} +  \|u_{xy}\|_2^{1/2} \right)   \|u_y\|_2^{1/2} \left (  \|u_y\|_2^{1/2} +  \|u_{xy}\|_2^{1/2} \right)  \notag\\
&\leq  C \|u_x\|_2  \|u_y\|_2  \left (  \|u_y\|_2+  \|u_{xy}\|_2\right)  \notag\\
&\leq C \|u_x\|_2  \|u_y\|_2  \|u_{xy}\|_2 + C \|u_x\|_2  \|u_y\|_2^2
\notag\\
&\leq \epsilon  \|u_{xy}\|_2^2 + C \|u_x\|_2^2  \|u_y\|_2^2 + C \|u_y\|_2^2.
\end{align}

Combining (\ref{lo-1}), (\ref{be-5}) and (\ref{be-6}) yields
\begin{align}  \label{be-7}
\left|\int_{\Omega}  (u^2)_{xxy}  \phi_{x} dx dy \right| \leq  \epsilon \|u_{xy}\|_2^2 + C  \|u_y\|_2^2   \left(\|u\|_2^2 \|u_x\|_2^2   +  \|u\|_2^4   + \|u_x\|_2^2   + 1\right).
\end{align}

Next, we consider the integral

\begin{align} \label{be-8}
&\int_{\Omega}  (uv)_{xyy} \phi_{x} dx dy
 =   \int_{\Omega} ( -u u_x + u_y v ) \phi_{xxy} dx dy.
\end{align}
Here,
\begin{align}    \label{be-9}
 &\left|\int_{\Omega} u_y v \phi_{xxy} dx dy \right|\notag\\
 &\leq C \|u_y\|_2^{1/2} (\|u_y\|_2^{1/2} +  \|u_{xy}\|_2^{1/2}   )   \|v\|_2^{1/2} (\|v\|_2^{1/2} +  \|v_y\|_2^{1/2}   ) \|\phi_{xxy}\|_2 \notag\\
 &\leq C \|u_y\|_2^{1/2} (\|u_y\|_2^{1/2} +  \|u_{xy}\|_2^{1/2}   )   \|v\|_2^{1/2} (\|v\|_2^{1/2} +  \|v_y\|_2^{1/2}   ) \|u_{xy}\|_2 \notag\\
 &\leq C  \|u_y\|_2^{1/2} \|v\|_2^{1/2}  \|v_y\|_2^{1/2}   \|u_{xy}\|_2^{3/2} + C  \|u_y\|_2^{1/2} \|v\|_2  \|u_{xy}\|_2^{3/2}  \notag\\
 & \hspace{0.2 in} +  C  \|u_y\|_2 \|v\|_2^{1/2}  \|v_y\|_2^{1/2}   \|u_{xy}\|_2 +   C  \|u_y\|_2 \|v\|_2  \|u_{xy}\|_2  \notag\\
 &\leq  \epsilon \|u_{xy}\|_2^2 + C\left(\|u_y\|_2^2 \|v\|_2^2  \|u_x\|_2^2 + \|u_y\|_2^2 \|v\|_2^4 +   \|u_y\|_2^2 \|u_x\|_2^2 +  \|u_y\|_2^2 \|v\|_2^2\right). 
  \end{align}
Also,
\begin{align}     \label{be-10}
&\left|\int_{\Omega} u u_x \phi_{xxy} dx dy\right|    =   \frac{1}{2} \left|\int_{\Omega} (u^2)_x \phi_{xxy} dx dy\right|    =      \frac{1}{2} \left|\int_{\Omega} (u^2)_{xxy} \phi_{x} dx dy\right|    \notag\\
 &\leq  \epsilon \|u_{xy}\|_2^2 + C  \|u_y\|_2^2   \left(\|u\|_2^2 \|u_x\|_2^2   +  \|u\|_2^4   + \|u_x\|_2^2   + 1\right),
 \end{align}
 due to (\ref{be-7}).

Combining (\ref{be-8})-(\ref{be-10}) yields
\begin{align}  \label{be-11}
\left|\int_{\Omega}  (uv)_{xyy} \phi_{x} dx dy \right| \leq  \epsilon \|u_{xy}\|_2^2 + C  \|u_y\|_2^2   \left(\|\mathbf u\|_2^2 \|u_x\|_2^2   +  \|\mathbf u\|_2^4   + \|u_x\|_2^2   + 1\right).
\end{align}

Moreover,
\begin{align} \label{be-12}
&\int_{\Omega} [  (v^2)_{yyy}] \phi_{x} dx dy
= 2 \int_{\Omega}                      (v v_y)_{yy}   \phi_{x} dx dy  \notag\\
& =  -2 \int_{\Omega}                      (v u_x)_{yy}  \phi_{x} dx dy
= -2  \int_{\Omega}                      (v_y u_x + v u_{xy})_{y}  \phi_{x} dx dy.
\end{align}

Notice that
\begin{align} \label{be-13}
&     \left|  \int_{\Omega}                      (v_y u_x)_y   \phi_{x} dx dy       \right| =      \left|  \int_{\Omega}                      (u_x^2)_y   \phi_{x} dx dy       \right|     =     2  \left|\int_{\Omega}         u_x u_{xy}         \phi_{x} dx dy  \right| \notag\\
&\leq C \|u_x\|_2^{1/2}  \left(   \|u_x\|_2^{1/2}  +     \|u_{xy}\|_2^{1/2} \right) \|u_{xy}\|_2 \|\phi_{x}\|_2^{1/2}  \left(  \|\phi_{x}\|_2^{1/2}  +   \|\phi_{xx}\|_2^{1/2}      \right)  \notag\\
&\leq C \|u_x\|_2^{1/2}  \left(   \|u_x\|_2^{1/2}  +     \|u_{xy}\|_2^{1/2} \right) \|u_{xy}\|_2 \|u\|_2^{1/2}  \left(  \|u\|_2^{1/2} +   \|u_y\|_2^{1/2}      \right) \notag\\
&\leq  C \|u_x\|_2^{1/2} \|u_{xy}\|_2^{3/2}  \|u\|_2^{1/2}  \|u_y\|_2^{1/2}  +  C \|u_x\|_2^{1/2} \|u_{xy}\|_2^{3/2}  \|u\|_2 \notag\\
& \hspace{0.2 in} + C \|u_x\|_2 \|u_{xy}\|_2 \|u\|_2^{1/2}  \|u_y\|_2^{1/2} + C \|u_x\|_2 \|u_{xy}\|_2 \|u\|_2      \notag\\
&\leq \epsilon \|u_{xy}\|_2^2 + C \|u_x\|_2^2   \|u\|_2^2 \|u_y\|_2^2 + C \|u_x\|_2^2   \|u\|_2^4 + C \|u_x\|_2^2 \|u_y\|_2^2 + C \|u_x\|_2^2 \|u\|_2^2.
\end{align}
Moreover,
\begin{align}     \label{be-14}
&\left|  \int_{\Omega}                      (v u_{xy})_{y}  \phi_{x} dx dy \right| = \left|\int_{\Omega}    v u_{xy}      \phi_{xy} dx dy  \right|     \notag\\
&\leq C \|v\|_2^{1/2} (\|v\|_2^{1/2} +  \|v_y\|_2^{1/2}  )  \|u_{xy}\|_2   \|\phi_{xy}\|_2^{1/2}  (   \|\phi_{xy}\|_2^{1/2}   +  \|\phi_{xxy}\|_2^{1/2}       )   \notag\\
&\leq C \|v\|_2^{1/2} (\|v\|_2^{1/2} +  \|v_y\|_2^{1/2}  )  \|u_{xy}\|_2   \|u_{y}\|_2^{1/2}  (   \|u_{y}\|_2^{1/2}   +  \|u_{xy}\|_2^{1/2}       )  \notag\\
&\leq  C  \|v\|_2^{1/2} \|v_y\|_2^{1/2}  \|u_y\|_2^{1/2}  \|u_{xy}\|_2^{3/2} + C \|v\|_2 \|u_y\|_2^{1/2}  \|u_{xy}\|_2^{3/2}  \notag\\
& \hspace{0.2 in}  + C  \|v\|_2^{1/2} \|v_y\|_2^{1/2}  \|u_y\|_2 \|u_{xy}\|_2 + C  \|v\|_2   \|u_y\|_2 \|u_{xy}\|_2    \notag\\
&\leq \epsilon \|u_{xy}\|_2^2 + C\left( \|v\|_2^2 \|u_x\|_2^2  \|u_y\|_2^2 +   \|v\|_2^4 \|u_y\|_2^2  +  \|u_x\|_2^2  \|u_y\|_2^2 +     \|v\|_2^2  \|u_y\|_2^2 \right).
\end{align}

Combining (\ref{be-12})-(\ref{be-14}) yields
\begin{align}  \label{be-15}
&\left|\int_{\Omega}   (v^2)_{yyy} \phi_{x} dx dy  \right| \notag\\
&\leq  \epsilon \|u_{xy}\|_2^2 + C  \|u_y\|_2^2 \left(   \|\mathbf u\|_2^2 \|u_x\|_2^2 +  \|u_x\|_2^2      +     \|v\|_2^4 +    \|v\|_2^2   \right) + C \|u_x\|_2^2 ( \|u\|_2^4 +  \|u\|_2^2    ).
\end{align}

Now, due to (\ref{MAIN}), (\ref{be-7}), (\ref{be-11}) and (\ref{be-15}), we obtain
\begin{align} \label{e-20}
\frac{d}{dt} \|u_{y}\|_2^2 + \|u_{xy}\|_2^2 &\leq  C  \|u_y\|_2^2   \left(\|\mathbf u\|_2^2 \|u_x\|_2^2   +  \|\mathbf u\|_2^4   + \|u_x\|_2^2   + 1\right) \notag\\
 &+     C \|u_x\|_2^2 ( \|u\|_2^4 +  \|u\|_2^2    ) + \|\partial_y f_1\|_2^2.
\end{align}
Therefore, for all $t\geq 0$, 
\begin{align} \label{e-21}
\|u_{y}(t)\|_2^2 &\leq  e^{C\int_0^t   \left(\|\mathbf u\|_2^2 \|u_x\|_2^2   +  \|\mathbf u\|_2^4   + \|u_x\|_2^2   + 1\right) ds    }    \notag\\
& \times \Big(   \|\partial_y u_0\|_2^2 + C\int_0^t   \|u_x\|_2^2 ( \|u\|_2^4 +  \|u\|_2^2 ) ds  + \int_0^t    \|\partial_y f_1\|_2^2 ds   \Big).
\end{align}
For every $t>0$, the right-hand side of (\ref{e-21}) is finite and has an upper bound depending on $\|\mathbf u_0\|_2$, $\|\partial_y u_0\|_2$, $\int_0^t \|\mathbf f\|_2 ds$, $\int_0^t \|\partial_y f_1\|_2^2 ds$
and the time $t$.

Due to (\ref{e-20}) and (\ref{e-21}), we get, for all $t\geq 0$,
\begin{align} \label{e-22}
\int_0^t  \|u_{xy}\|_2^2 ds \leq C(\|\mathbf u_0\|_2,   \|\partial_y u_0\|_2,  \|\mathbf f\|_2,  \|\partial_y f_1\|_2, t  ).
\end{align}

\smallskip

\subsubsection{Rigorous justification of existence of weak solutions} \label{rigor}
Assume initial velocity $\mathbf u_0 = (u_0,v_0) \in \mathcal H$ with $\partial_y u_0 \in L^2(\Omega)$. According to space $\mathcal H$ defined in (\ref{space}), we write 
$u_0=\sum_{\mathbf k \in \mathbb Z^2} a_{\mathbf k} e^{2 i \pi \mathbf k \cdot \mathbf x}$
and 
$v_0=\sum_{k_1 \in \mathbb Z, \, k_2\in \mathbb Z^+} b_{\mathbf k} e^{2i \pi k_1 x} \sin (\pi k_2 y)$. We consider a sequence of initial data $\mathbf u_{0,m} = (u_{0,m}, v_{0,m})$ defined as
\begin{align} \label{square-p}
u_{0,m}=\sum_{   \substack{ |k_1| \leq m \\ |k_2| \leq m} } a_{\mathbf k} e^{2 i \pi \mathbf k \cdot \mathbf x},     \;\;  v_{0,m}=\sum_{    \substack{  |k_1|\leq m \\ 1 \leq k_2 \leq m}} b_{\mathbf k} e^{2i \pi k_1 x} \sin (\pi k_2 y).
\end{align}
It follows that $\mathbf u_{0,m} \rightarrow \mathbf u_0$ in $L^2(\Omega)$ and $\partial_y (u_{0,m}) \rightarrow \partial_y u_0$ in $L^2(\Omega)$.

Also, assume that $\mathbf f = (f_1,f_2)  \in L^1(0,T; \mathcal H)$ with $\partial_y f_1\in L^2(0,T; L^2(\Omega))$ for any $T>0$. Then, we can write $f_1=\sum_{\mathbf k \in \mathbb Z^2} \alpha_{\mathbf k}(t) e^{2 i \pi \mathbf k \cdot \mathbf x}$
and 
$f_2=\sum_{k_1 \in \mathbb Z, \, k_2\in \mathbb Z^+} \beta_{\mathbf k}(t) e^{2i \pi k_1 x} \sin (\pi k_2 y)$. We consider a sequence of external force $\mathbf f_m = (f_{1,m}, f_{2,m})$ defined as
\begin{align} 
f_{1,m}=\sum_{   \substack{ |k_1| \leq m \\ |k_2| \leq m} } \alpha_{\mathbf k}(t) e^{2 i \pi \mathbf k \cdot \mathbf x},     \;\;  f_{2,m}=\sum_{    \substack{  |k_1|\leq m \\ 1 \leq k_2 \leq m}} \beta_{\mathbf k}(t) e^{2i \pi k_1 x} \sin (\pi k_2 y).
\end{align}
It follows that $\mathbf f_m \rightarrow \mathbf f$ in $L^1(0,T; \mathcal H)$.

Under initial condition $\mathbf u_m(0)=\mathbf u_{0,m}$ and external force $\mathbf f_m$, we consider the equation 
\begin{align}  \label{Gal}
\partial_t \mathbf u_m-  \partial_{xx}\mathbf u_m + (\mathbf u_m\cdot \nabla)\mathbf u_m  + \nabla p_m  = \mathbf f_m(x,y,t), 
\end{align}
where $\nabla \cdot \mathbf u_m =0$. Since $\mathbf u_{0,m}$ and $\mathbf f_m$ have only finitely many Fourier modes, they are real analytic functions. 
We remark that equation (\ref{Gal}) is more regular than 2D Euler equations, and therefore the global well-posedness of 2D Euler equations with analytic initial data and analytic forcing terms also holds for equation (\ref{Gal}). 
Thus, we obtain that equation (\ref{Gal}) has a global analytic solution $\mathbf u_m$ for each positive integer $m$. Furthermore, all the \emph{a priori} estimates conducted in subsections \ref{L2est} and \ref{uyest} are valid for $\mathbf u_m$.

Let us fix an arbitrary $T>0$.  Due to \emph{a priori} estimates (\ref{e-1})-(\ref{e-2}) and (\ref{e-21}), we have
\begin{align}
&\|\mathbf u_m (t)\|_2^2     +     \int_0^t  \|\partial_x \mathbf u_m\|_2^2 ds   \leq    C\|\mathbf u_0\|_2^2 + C\Big(\int_0^T  \|\mathbf f\|_2 ds\Big)^2 \;  ;    \label{mb-1}\\
&\|\partial_y u_m(t)\|_2^2  \leq   C(  \|\mathbf u_0\|_2, \|\partial_y u_{0}\|_2,  \|\mathbf f\|_2,  \|\partial_y f_1\|_2  , T), \label{mb-3}
\end{align}
for $0\leq t \leq T$, for any $m\in \mathbb Z^+$, where we have used that $\|\mathbf u_{0,m}\|_2 \leq \|\mathbf u_0\|_2$ and $\|\mathbf f_m\|_2 \leq \|\mathbf f\|_2$.

Due to (\ref{mb-1})-(\ref{mb-3}), we have 
\begin{align}
&\mathbf u_m  \;\; \text{is uniformly bounded in }\; L^{\infty}(0,T; \mathcal H);  \label{uni-1}\\
&\partial_y u_m \;\; \text{is uniformly bounded in }\; L^{\infty}(0,T; L^2(\Omega));  \label{uni-1-1}\\
&\partial_x \mathbf u_m \;\; \text{is uniformly bounded in }\; L^2(0,T; \mathcal H).   \label{uni-2}
\end{align}

Recall $\mathcal V = \mathcal H \cap (H^1(\Omega))^2$. 
Using the fact that $\partial_x u_m = \partial_y v_m$ and the uniform bounds (\ref{uni-1})-(\ref{uni-2}), we obtain 
\begin{align}
&\mathbf u_m  \;\; \text{is uniformly bounded in }\; L^2(0,T; \mathcal V).  \label{uni-2-1}
\end{align}

For any vector field $\phi \in L^2( 0,T; \mathcal V )$, using integration by parts, the Cauchy–Schwarz inequality, and Ladyzhenskaya's inequality, we have
\begin{align*}
&\int_0^T \int_{\Omega} \left[ (\mathbf u_m\cdot \nabla)\mathbf u_m \right] \cdot \phi dx dy dt
\leq C \int_0^T \left(\|u_m\|_4^2 +   \|v_m\|_4^2\right)  \|\nabla \phi\|_2 dt  \notag\\
& \leq C \int_0^T \|\mathbf u_m\|_2 \|\mathbf u_m\|_{H^1}  \|\nabla \phi\|_2 dt  \notag\\
&\leq C(\sup_{t\in [0,T]} \|\mathbf u_m\|_2)  \Big(\int_0^T \|\mathbf u_m\|_{H^1}^2 dt\Big)^{1/2} \Big(\int_0^T \|\nabla \phi\|_2^2 dt \Big)^{1/2} \notag\\
&\leq M \Big(\int_0^T \|\nabla \phi\|_2^2 dt \Big)^{1/2},
\end{align*}
where we have used uniform bounds (\ref{uni-1}) and (\ref{uni-2-1}). Therefore, $(\mathbf u_m\cdot \nabla)\mathbf u_m $ is uniformly bounded in $L^2(0,T; \mathcal V')$. Then, it follows from (\ref{Gal}) and (\ref{Delta-p}) that 
\begin{align}
\partial_t \mathbf u_m \;\; \text{is uniformly bounded in }\; L^2(0,T; \mathcal V'). \label{uni-3}
\end{align}

Because of the uniform bounds (\ref{uni-1}), (\ref{uni-2-1}) and (\ref{uni-3}), there exists a vector field $\mathbf u$ such that, on a subsequence, which is still denoted by $\mathbf u_m$, the following weak convergences hold:
\begin{align}
&\mathbf u_m \longrightarrow \mathbf u \;\; \text{weakly$^*$ in} \;   L^{\infty}(0,T;  \mathcal H)  ;    \label{con-1} \\
&\nabla \mathbf u_m \longrightarrow \nabla \mathbf u \;\;  \text{weakly in} \;   L^2(0,T; L^2(\Omega));  \label{con-3} \\
& \partial_t \mathbf u_m      \longrightarrow \partial_t \mathbf u \;\;  \text{weakly$^*$ in} \;    L^2(0,T; \mathcal V' ).  \label{con-4}
\end{align}

Moreover, due to the uniform bounds (\ref{uni-2-1}) and (\ref{uni-3}), and thanks to Aubin's Compactness Theorem, we obtain the following strong convergence of a subsequence:
\begin{align}  \label{con-strong}
\mathbf u_m \longrightarrow \mathbf u  \;\;  \text{strongly in} \;       L^2(0,T; \mathcal H). 
\end{align}

Also, we note that, since $\mathbf u \in L^2(0,T;\mathcal V)$ and $\mathbf u_t\in L^2(0,T; \mathcal V' )$, we obtain the continuity in time: $\mathbf u \in C([0,T]; \mathcal H)$.

We aim to justify that the limit function $\mathbf u$ is indeed a weak solution of equation (\ref{PDE}). 
Let $\mathbf w \in \mathcal H$ be a trigonometric polynomial. 
Then, from (\ref{Gal}), we have
\begin{align}  \label{before-l}
&\int_0^T \langle \partial_t \mathbf u_m, \mathbf w \rangle_{\mathcal V' \times \mathcal V} dt + \int_0^T (\partial_x \mathbf u_m, \mathbf w_x)_{L^2} dt + \int_0^T \int_{\Omega}  [(\mathbf u_m\cdot \nabla)\mathbf u_m] \cdot \mathbf w dx dy dt  \notag\\
&= \int_0^T (\mathbf f_m, \mathbf w)_{L^2} dt. 
\end{align}
We focus on the nonlinear term:
\begin{align}  \label{non}
& \int_0^T  \int_{\Omega}  \left[ (\mathbf u_m\cdot \nabla)\mathbf u_m \right] \cdot  \mathbf w dx dy dt  \notag\\
& =  \int_0^T  \int_{\Omega}  \left[ ((\mathbf u_m - \mathbf u)\cdot \nabla )\mathbf u_m \right] \cdot  \mathbf w dx dy dt +  \int_0^T  \int_{\Omega}  \left[ (\mathbf u\cdot \nabla)\mathbf u_m \right] \cdot  \mathbf w dx dy dt.
\end{align}
Notice that
\begin{align}  \label{non-1}
& \left| \int_0^T  \int_{\Omega}  \left[ ((\mathbf u_m - \mathbf u)\cdot \nabla) \mathbf u_m \right]  \cdot \mathbf w dx dy dt \right|  \notag\\
&\leq   \Big(\int_0^T \int_{\Omega} |\mathbf u_m - \mathbf u|^2 dx dy dt \Big)^{1/2}      \Big(\int_0^T  \int_{\Omega} |\nabla \mathbf u_m |^2 dx dy dt \Big)^{1/2} \|\mathbf w\|_{L^{\infty}(\Omega)} \notag\\
&\leq   C   \Big(\int_0^T  \int_{\Omega} |\mathbf u_m - \mathbf u|^2 dx dy dt \Big)^{1/2}   \|\mathbf w\|_{L^{\infty}(\Omega)}  \longrightarrow 0,
 \end{align}
due to the strong convergence (\ref{con-strong}), as well as the uniform bound (\ref{uni-2-1}).

Moreover, applying the weak convergence (\ref{con-3}), and the assumption that $\mathbf w \in \mathcal H$ is a trigonometric polynomial, we obtain  
\begin{align} \label{non-2}
\lim_{m\rightarrow \infty} \int_0^T  \int_{\Omega}  \left[ (\mathbf u\cdot \nabla)\mathbf u_m \right] \cdot  \mathbf w dx dy dt =  \int_0^T  \int_{\Omega}  \left[ (\mathbf u\cdot \nabla)\mathbf u \right] \cdot  \mathbf w dx dy dt.
\end{align}
Combining (\ref{non})-(\ref{non-2}) yields
\begin{align}  \label{non-3}
\lim_{m\rightarrow \infty} \int_0^T  \int_{\Omega}  \left[ (\mathbf u_m \cdot \nabla)\mathbf u_m \right] \cdot  \mathbf w dx dy dt  =  \int_0^T  \int_{\Omega}  \left[ (\mathbf u\cdot \nabla)\mathbf u \right] \cdot  \mathbf w dx dy dt,
\end{align}
for any trigonometric polynomial $\mathbf w \in \mathcal H$.
 
 Because of the weak convergence (\ref{con-3})-(\ref{con-4}) and the limit (\ref{non-3}), we can let $m\rightarrow \infty$ in (\ref{before-l}) to conclude that
 \begin{align}  \label{non-4}
&\int_0^T \langle \partial_t \mathbf u, \mathbf w \rangle_{\mathcal V' \times \mathcal V} dt + \int_0^T (\partial_x \mathbf u, \mathbf w_x)_{L^2} dt + \int_0^T \int_{\Omega}  [(\mathbf u\cdot \nabla)\mathbf u] \cdot \mathbf w dx dy dt  \notag\\
&= \int_0^T (\mathbf f, \mathbf w)_{L^2} dt,
\end{align} 
for any trigonometric polynomial $\mathbf w \in \mathcal H$, for any $T>0$. Finally, due to the fact that trigonometric polynomials are dense in $\mathcal V$, 
then using a straightforward density argument, we obtain that (\ref{non-4}) holds for all $\mathbf w \in \mathcal V$, for any $T>0$.

\bigskip

\subsection{Uniqueness of weak solutions}

Let $\mathbf u_1$ and $\mathbf u_2$ be two weak solutions found in the previous subsection.  
Consider the difference $\mathbf u = \mathbf u_1 - \mathbf u_2$. Write $\mathbf u =(u,v)$, $\mathbf u_1 =(u_1,v_1)$ and $\mathbf u_2 =(u_2,v_2)$. 
We have 
\begin{align}    \label{u-1}
\mathbf u_t - \mathbf u_{xx} + (\mathbf u\cdot \nabla)\mathbf u_1 +   (\mathbf u_2\cdot \nabla)\mathbf u +  \nabla p =0,
\end{align}
with $\nabla\cdot  \mathbf u = u_x + v_y=0$.

We take the inner product of (\ref{u-1}) with $\mathbf u$. Due to identity (\ref{vanish}), we have
\begin{align}      \label{u-2}
&\frac{1}{2}\frac{d}{dt} \|\mathbf u\|_2^2 + \|\mathbf u_x\|_2^2 = - \left((\mathbf u\cdot \nabla)\mathbf u_1, \mathbf u \right)  \notag\\
&= -  \int_{\Omega} \left(u \partial_x u_1 +   v \partial_y u_1\right) u dx dy   -   \int_{\Omega}  \left(u \partial_x v_1 +  v \partial_y v_1 \right) v dx dy.
\end{align}

In the following, we estimate all terms on the right-hand side of (\ref{u-2}).

First, by using inequality (\ref{ineq2}), we deduce
\begin{align}     \label{u-3}
&\left|\int_{\Omega} u^2 \partial_x u_1 dx dy \right|= 2 \left|\int_{\Omega} u u_x  u_1 dx dy\right| \leq 2 \|u\|_2 \|u_x\|_2 \|u_1\|_{\infty} \notag\\
&\leq C  \|u\|_2 \|u_x\|_2    \left(   \|u_1\|_2 +     \|\partial_{x}u_1\|_2       +       \|\partial_{y }u_1\|_2     +             \|\partial_{xy }u_1\|_2          \right) \notag\\
& \leq \epsilon \|u_x\|_2^2 + C \|u\|_2^2 \left(   \|u_1\|_2^2 +     \|\partial_{x}u_1\|_2^2       +       \|\partial_{y }u_1\|_2^2     +             \|\partial_{xy }u_1\|_2^2    \right).
\end{align}

Next, we look at
\begin{align}     \label{u-4}
\left|\int_{\Omega} u v \partial_x v_1 dx dy \right|
=  \left|\int_{\Omega} (u_x v + u v_x) v_1 dx dy \right|.
\end{align}
Using inequalities (\ref{ineq}) and (\ref{Poin}) as well as the fact that $v_y = -u_x$, we estimate
\begin{align}     \label{u-5}
\left|\int_{\Omega} u_x v  v_1 dx dy \right|  &\leq C\|u_x\|_2 \|v\|_2^{1/2}   \|v_y\|_2^{1/2}  \|v_1\|_2^{1/2}  \left( \|v_1\|_2^{1/2} +  \|\partial_x v_1\|_2^{1/2}\right) \notag\\
&=  C \|u_x\|_2^{3/2} \|v\|_2^{1/2} \|v_1\|_2^{1/2} \|\partial_x v_1\|_2^{1/2}  + C  \|u_x\|_2^{3/2} \|v\|_2^{1/2} \|v_1\|_2   \notag  \\
&\leq \epsilon  \|u_x\|_2^2 + C \|v\|_2^2 \left( \|v_1\|_2^2\|\partial_x v_1\|_2^2 + \|v_1\|_2^4   \right).
\end{align}
Similarly, due to inequalities (\ref{ineq}) and (\ref{Poin}), we deduce 
\begin{align}    \label{u-6}
\left|\int_{\Omega} u v_x  v_1 dx dy\right| &\leq C\|v_x\|_2 \|u\|_2^{1/2}  \left(  \|u\|_2^{1/2}  +    \|u_x\|_2^{1/2} \right) \|v_1\|_2^{1/2}     \|\partial_y v_1\|_2^{1/2}  \notag\\
&=  C \|\mathbf u_x\|_2^{3/2} \|u\|_2^{1/2} \|v_1\|_2^{1/2} \|\partial_x u_1\|_2^{1/2}   +  C \|v_x\|_2 \|u\|_2 \|v_1\|_2^{1/2} \|\partial_x u_1\|_2^{1/2}     \notag    \\
&\leq \epsilon  \|\mathbf u_x\|_2^2 + C \|u\|_2^2 \left(\|v_1\|_2^2 \|\partial_x u_1\|_2^2   +    \|v_1\|_2^2 +   \|\partial_x u_1\|_2^2     \right).
\end{align}
Combining (\ref{u-4})-(\ref{u-6}) yields
\begin{align}     \label{u-7}
\left|\int_{\Omega} u v \partial_x v_1 dx dy\right| \leq  2\epsilon  \|\mathbf u_x\|_2^2 + C \|\mathbf u\|_2^2   \left(    \|v_1\|_2^2\|\partial_x \mathbf u_1\|_2^2 +   \|v_1\|_2^4  + \|v_1\|_2^2  +     \|\partial_x u_1\|_2^2  \right).
\end{align}

Next, we apply inequalities (\ref{ineq}) and (\ref{Poin}) again to derive
\begin{align}   \label{u-8}
&\left|\int_{\Omega} v^2 \partial_y v_1 dx dy\right| = 2 \left|\int_{\Omega} v v_y v_1 dx dy \right|= 2 \left|\int_{\Omega} v u_x v_1 dx dy\right|  \notag \\
&\leq C \|u_x\|_2 \|v\|_2^{1/2}     \|v_y\|_2^{1/2}  \|v_1\|_2^{1/2}   \left(   \|v_1\|_2^{1/2}   +    \|\partial_x v_1\|_2^{1/2}   \right)   \notag\\
&= C \|u_x\|_2^{3/2}  \|v\|_2^{1/2}  \|v_1\|_2^{1/2} \|\partial_x v_1\|_2^{1/2}   +    C \|u_x\|_2^{3/2}  \|v\|_2^{1/2}  \|v_1\|_2         \notag\\
&\leq \epsilon \|u_x\|_2^2 + C  \|v\|_2^2  \left(   \|v_1\|_2^2 \|\partial_x v_1\|_2^2 +    \|v_1\|_2^4 \right).
\end{align}

Finally, using the fact that $H^1$ is imbedded in $L^{\infty}$ in 1D, we estimate
\begin{align}     \label{u-9}
&\left|\int_{\Omega} v (\partial_y u_1) u dx dy\right| \leq \int_0^1 \Big(\sup_{y\in [0,1]} |v| \Big) \Big(\int_0^1 |\partial_y u_1|^2 dy\Big)^{1/2}   \Big(\int_0^1 u^2 dy\Big)^{1/2}   dx  \notag\\
&\leq C \int_0^1   \Big(\int_0^1  v_y^2 dy\Big)^{1/2}    \Big(\int_0^1 |\partial_y u_1|^2 dy\Big)^{1/2}   \Big(\int_0^1 u^2 dy\Big)^{1/2}   dx  \notag\\
&\leq C    \|v_y\|_2  \|u\|_2     \Big[  \sup_{x\in [0,1]} \int_0^1 |\partial_y u_1|^2 dy  \Big]^{1/2}   \notag\\
&\leq C \|u_x\|_2 \|u\|_2  \left(  \|\partial_y u_1\|_2^2    +     \|\partial_{xy} u_1\|_2^2 \right)^{1/2}        \notag\\
&\leq   \epsilon \|u_x\|_2^2 + C \|u\|_2^2  \left(  \|\partial_y u_1\|_2^2    +     \|\partial_{xy} u_1\|_2^2 \right).
\end{align}

Now, we can apply estimates (\ref{u-3}), (\ref{u-7}), (\ref{u-8}) and (\ref{u-9}) to the right-hand side of (\ref{u-2}). Then
\begin{align*}
\frac{d}{dt} \|\mathbf u\|_2^2 + \|\mathbf u_x\|_2^2 \leq C \|\mathbf u\|_2^2 (  \|v_1\|_2^4 +    \|\mathbf u_1\|_2^2 +     \|\partial_{x} u_1\|_2^2  +      \|\partial_y u_1\|_2^2        +          \|\partial_{xy} u_1\|_2^2 + \|v_1\|_2^2 \|\partial_x \mathbf u_1\|_2^2 ).
\end{align*}
By the Gr\"onwall inequality, we have
\begin{align}  \label{u-10}
\|\mathbf u(t)\|_2^2 \leq  \|\mathbf u_0\|_2^2  e^{C \int_0^t  ( \|v_1\|_2^4 +    \|\mathbf u_1\|_2^2 +     \|\partial_{x} u_1\|_2^2  +      \|\partial_y u_1\|_2^2        +          \|\partial_{xy} u_1\|_2^2 + \|v_1\|_2^2 \|\partial_x \mathbf u_1\|_2^2 )  ds }.
\end{align}
Due to (\ref{e-1})-(\ref{e-2}) and (\ref{e-21})-(\ref{e-22}), the right-hand side of (\ref{u-10}) is finite for all $t>0$. This implies the uniqueness of weak solutions.

\bigskip

\section{Uniform boundness of $\|\mathbf u\|_{H^2(\Omega)}$}

This section is devoted to proving Theorem \ref{thm-3}: the uniform boundedness of $H^2$-norm of the velocity $\mathbf u$.

Consider the vorticity $\omega = v_x - u_y$. Then the vorticity formulation of equation (\ref{PDE}) reads
\begin{align} \label{vor}
\omega_t + u \omega_x + v \omega_y = \omega_{xx} + \text{curl} \, \mathbf f.
\end{align}

For any function $g\in L^1(\Omega)$, we recall the notation
\begin{align*} 
\bar{g}(y) = \int_0^1 g(x,y) dx,  \qquad \tilde{g}(x,y) = g(x,y)- \bar{g}(y).
\end{align*}
Clearly, $\tilde g_x = g_x$. Also, we have the Poincar\'e inequality:
\begin{align} \label{tilde}
\|\tilde g\|_2 \leq C \|\tilde g_x\|_2,
\end{align}
since the horizontal mean of $\tilde g$ is zero. 

Below we shall conduct the estimate.

\subsection{$\|\omega\|_2^2$-estimate}
Taking the inner product of (\ref{vor}) with $\omega$, we obtain
\begin{align} \label{v-1}
\frac{1}{2} \frac{d}{dt} \|\omega\|_2^2 + \|\omega_x\|_2^2 = (\text{curl} \, \mathbf f, \omega)_{L^2} \leq  \|\text{curl} \, \mathbf f\|_2 \|\omega\|_2.
\end{align}

Therefore, $\frac{d}{dt} \|\omega\|_2 \leq  \|\text{curl} \, \mathbf f\|_2$. It follows that
\begin{align}     \label{v-2}
\|\omega(t)\|_2 \leq \|\omega_0\|_2 + \int_0^t \|\text{curl} \, \mathbf f\|_2 ds
\end{align}
Then, by integrating (\ref{v-1}) over $[0,t]$ and using (\ref{v-2}), we obtain
\begin{align}      \label{v-2-0}
& \int_0^t \|\omega_x\|_2^2 ds \leq \frac{1}{2} \|\omega_0\|_2^2 + \int_0^t  \|\text{curl} \, \mathbf f\|_2 \|\omega\|_2 ds  \notag\\
&\leq  \frac{1}{2} \|\omega_0\|_2^2 + \Big(\int_0^t  \|\text{curl} \, \mathbf f\|_2  ds \Big)    \Big( \|\omega_0\|_2 + \int_0^t \|\text{curl} \, \mathbf f\|_2 ds \Big)     \notag\\
&\leq    \|\omega_0\|_2^2 + 2 \Big(\int_0^t  \|\text{curl} \, \mathbf f\|_2  ds \Big)^2. 
\end{align}

Due to $u_x + v_y =0$, it is easy to see that 
\begin{align} \label{v-2-1}
\|\omega\|_2 = \|\nabla \mathbf u\|_2.
\end{align}
Indeed,
\begin{align*} 
&\|\omega\|_2^2 = \|v_x - u_y\|_2^2 = \int_{\Omega} \left(v_x^2 - 2 v_x u_y + u_y^2\right) dx dy \notag\\
& = \|v_x\|_2^2 + \|u_y\|_2^2 - 2\int_{\Omega} v_y u_x dx dy = \|v_x\|_2^2 + \|u_y\|_2^2 + \|u_x\|_2^2 + \|v_y\|_2^2 = \|\nabla \mathbf u\|_2^2,
\end{align*}
where we have used the boundary condition $v|_{y=0} = v|_{y=1} =0$ to conduct integration by parts. Similarly, we have $\|\omega_x\|_2 = \|\nabla \mathbf u_x\|_2$.

Then, it follows from (\ref{v-2}), (\ref{v-2-0}) and (\ref{v-2-1}) that
\begin{align}   \label{v-2-2}
\|\nabla \mathbf u(t)\|_2^2 + \int_0^t \|\nabla \mathbf u_x\|_2^2 ds  \leq  C\|\omega_0\|_2^2 + C\Big(\int_0^t  \|\text{curl} \, \mathbf f\|_2  ds \Big)^2.
\end{align}
By (\ref{tilde}) and (\ref{v-2-2}), we obtain
\begin{align} \label{v-2-3}
\int_0^t \|\tilde{\mathbf u}\|_2^2 ds \leq C\int_0^t \|\nabla \tilde{\mathbf u}\|_2^2 ds \leq C \int_0^t \|\nabla \tilde{\mathbf u}_x\|_2^2 ds   \leq C\|\omega_0\|_2^2 + C\Big(\int_0^t  \|\text{curl} \, \mathbf f\|_2  ds \Big)^2.
\end{align}

Moreover, by (\ref{v-2-2}), we have
\begin{align}   \label{v-2-4}
\|\tilde{\mathbf u}\|_2^2 \leq C \|\tilde{\mathbf u}_x\|_2^2 = C \|\mathbf u_x\|_2^2 \leq C\|\omega_0\|_2^2 + C\Big(\int_0^t  \|\text{curl} \, \mathbf f\|_2  ds \Big)^2. 
\end{align}

\subsection{$\|\nabla \omega\|_2^2$-estimate}
Differentiate (\ref{vor}) by $x$ and multiply by $\omega_x$. Also, differentiate (\ref{vor}) by $y$ and multiply by $\omega_y$. Adding the results and integrating over $\Omega$, we obtain
\begin{align} \label{v-3}
&\frac{1}{2}  \frac{d}{dt} \|\nabla \om \|_2^2 \,   + \|\nabla \om_x\|_2^2 \notag\\
& =-  \int_{\Omega} \left[u_x \, \om_x \, \om_x +u_y  \, \om_y \, \om_x  + v_x \, \om_x \, \om_y +v_y \, \om_y \,  \om_y    -   (\text{curl} \, \mathbf f_x) \omega_x        -   (\text{curl} \, \mathbf f_y) \omega_y        \right] dxdy    \notag\\
& =-  \int_{\Omega} [ u_x  \, \om_x \, \om_x -\om \,  \om_y \, \om_x  + 2 v_x \,  \om_x \, \om_y +v_y \, \om_y \,  \om_y      -   (\text{curl} \, \mathbf f_x) \omega_x        -   (\text{curl} \, \mathbf f_y) \omega_y    ] dxdy.
\end{align}
.

The terms on the right-hand side can be estimated as follows. 

First, using integration by parts and inequality (\ref{ineq}), we obtain
\begin{align}  \label{v-4}
& \left|  \int_{\Omega}  u_x \,  \om_x \, \om_x \,  dxdy  \right| = \left|  \int_{\Omega}  \tilde{u}_x \, \om_x \, \om_x  \,   dxdy  \right| = 2 \left|  \int_{\Omega}  \tilde{u} \, \om_x \, \om_{xx}   \,   dxdy  \right|  \notag\\
& \leq C \left(\|\tilde{u} \|_2 +  \|\tilde{u} \|_2^{1/2}   \|\tilde{u}_y\|_2^{1/2} \right)  \|\om_x \|_2^{\frac{1}{2}} \|\om_{xx}\|_2^{\frac{3}{2}}  \notag\\
&\leq     \ee  \| \om_{xx} \|_2^2   +       C(\|\tilde{u} \|_2^4 +  \|\tilde{u} \|_2^2   \|\tilde{u}_y \|_2^2) \|\om_x \|_2^2.
 \end{align}

Using the boundary condition $v|_{y=0}=0$, we see that $\tilde v|_{y=0}=0$. Then, by inequality (\ref{Poin}), we have $\|\tilde v\|_2^2 \leq \|\tilde v_y\|_2^2$. Now, using inequalities (\ref{ineq}), we deduce
\begin{align}   \label{v-5}
& \left|  \int_{\Omega}  v_x \,  \om_x \, \om_y   \,   dxdy  \right| =  \left|  \int_{\Omega}  \tilde v_x \,  \om_x \, \om_y   \,   dxdy  \right|   \leq \left|  \int_{\Omega}  \tilde v \, \om_x \, \om_{xy}  \,    dxdy  \right| + \left|  \int_{\Omega}  \tilde v  \,  \om_y \, \om_{xx} \, dxdy
\right| \notag\\
& \leq        C\| \tilde v \|_2^{1/2}       \| \tilde v_y \|_2^{1/2} \left[   \|\om_x \|_2^{\frac{1}{2}} \| \om_{xx} \|_2^{\frac{1}{2}}  \, \| \om_{xy} \|_2
+  (\|\om_y \|_2 + \|\om_y \|_2^{\frac{1}{2}} \| \om_{xy} \|_2^{\frac{1}{2}} ) \, \| \om_{xx} \|_2 \right] \notag\\
& \leq     \ee \, \| \nabla \om_{x} \|_2^2    +    C  \|\tilde v\|_2 \| \tilde v_y \|_2 \left[   \|\om_x \|_2 \| \om_{xx} \|_2  \,
+  \|\om_y \|_2^2 + \|\om_y \|_2 \| \om_{xy} \|_2 \right]   \notag\\
& \leq    2 \ee \, \| \nabla \om_{x} \|_2^2   +   C \|\tilde v\|_2  \| \tilde v_y \|_2 \|\om_y\|_2^2 +   C \|\tilde v\|_2^2 \| \tilde v_y \|_2^2  \|\nabla \om \|_2^2 \notag\\
& \leq   2 \ee \, \| \nabla \om_{x} \|_2^2   +   C \| \tilde v_y \|_2^2 \|\om_y\|_2^2 +   C \|\tilde v\|_2^2 \| \tilde v_y \|_2^2  \|\nabla \om \|_2^2 \notag\\
& =  2 \ee \, \| \nabla \om_{x} \|_2^2   +   C \| \tilde u_x \|_2^2 \|\om_y\|_2^2 +   C   \|\tilde v\|_2^2  \| \tilde u_x \|_2^2  \|\nabla \om \|_2^2.
\end{align}

Next,
\begin{align}    \label{v-6}
& \left|  \int_{\Omega}  v_y \, \om_y \,  \om_y   \,  dxdy  \right| =  \left|  \int_{\Omega}  \tilde{u}_x \, \om_y \, \om_{y}   \,   dxdy  \right| =
2 \left|  \int_{\Omega}  \tilde{u} \,  \om_y \, \om_{xy} \, dxdy
\right|     \notag   \\
& \leq  C\left(\|\tilde{u} \|_2 +     \|\tilde u\|_2^{1/2} \| \tilde{u}_y \|_2^{1/2} \right)   \left(\|\om_y \|_2 + \|\om_y \|_2^{\frac{1}{2}} \| \om_{xy} \|_2^{\frac{1}{2}} \right) \, \| \om_{xy} \|_2 \notag\\
& \leq    \ee \| \om_{xy} \|_2^2   +    C  (\|\tilde{u} \|_2^2 +     \|\tilde u\|_2\| \tilde{u}_y \|_2 )   (\|\om_y \|_2^2 + \|\om_y \|_2 \| \om_{xy} \|_2 )        \notag  \\
& \leq        2 \ee \| \om_{xy} \|_2^2         +        C  (\|\tilde{u} \|_2^2 +  \| \tilde{u}_y \|_2^2+  \|\tilde{u} \|_2^4  +  \|\tilde u\|_2^2\| \tilde{u}_y \|_2^2    )   \|\om_y \|_2^2 .
\end{align}

Notice that $ \int_{\Omega}  \om \,  \bar{\om}_y \, \om_x  \, dxdy = \frac{1}{2} \int_{\Omega} (\om^2)_x  \,  \bar{\om}_y  dx dy =    - \frac{1}{2} \int_{\Omega} (\om^2)  \,  \bar{\om}_{xy}  dx dy  =    0 $. Thus
\begin{align}    \label{v-7}
&\left|  \int_{\Omega}  \om \,  \om_y \, \om_x  \, dxdy  \right| =   \left|  \int_{\Omega}  \om     \,  \tilde \om_y        \, \om_x      \,  dxdy  \right|  
   \notag\\
&\leq C\|\omega\|_2           \|\tilde \omega_{xy}\|_2      \left(\|\omega_x\|_2   +      \|\omega_x\|_2^{1/2}    \| \omega_{xy}\|_2^{1/2}       \right)        \notag\\
&\leq \ee  \| \om_{xy} \|_2^2   + C \left(\|\omega\|_2^4  +  \|\omega\|_2^2  \right)  \|\omega_x\|_2^2.
 \end{align}

In addition, due to the assumption that $\overline{f_1} =0$, we get
\begin{align*} 
\left| \int_{\Omega} \left[   (\text{curl} \, \mathbf f_x) \omega_x          +   (\text{curl} \, \mathbf f_y) \omega_y        \right] dxdy \right|
&\leq \int_{\Omega} \Big( |(\text{curl} \, \mathbf f) \omega_{xx}|  +   |(\partial_y f_2) \omega_{xy}|  
+   |(\partial_{yy} f_1) \tilde{\omega}_{y}|  \Big)    dx dy  \notag\\
&\leq C \| \nabla \om_{x} \|_2 \|\text{curl} \, \mathbf f\|_2 +   \|\partial_{yy} f_1\|_2 \|\tilde{\omega}_{y}\|_2,
\end{align*}
where we have used $\|\text{curl} \, \mathbf f\|_2 = \|\nabla \mathbf f\|_2$ since $\mathbf f$ is divergence free.

By the Poincar\'e inequality, we have 
\begin{align*} 
&\|\tilde{\omega}_{y}\|_2 \leq  C\|\tilde{\omega}_{xy}\|_2 =  C \| \omega_{xy} \|_2 \leq C \| \nabla \om_{x} \|_2.
\end{align*} 
Thus,
\begin{align} \label{v-f}
\left| \int_{\Omega} \left[   (\text{curl} \, \mathbf f_x) \omega_x          +   (\text{curl} \, \mathbf f_y) \omega_y        \right] dxdy \right| 
&\leq  C \| \nabla \om_{x} \|_2 ( \|\text{curl} \, \mathbf f\|_2 +   \|\partial_{yy} f_1\|_2 )  \notag\\
&\leq \epsilon \| \nabla \om_{x} \|_2^2 
+ C( \|\text{curl} \, \mathbf f\|_2^2 +   \|\partial_{yy} f_1\|_2^2 ).
\end{align}

Combining (\ref{v-3})-(\ref{v-f}) and taking $\ee$ sufficiently small, we obtain
\begin{align*}    
&\frac{d}{dt} \|\nabla \om \|_2^2 \,   + \|\nabla \om_x\|_2^2  \notag\\
&\leq    C(\|\tilde{u} \|_2^4 +  \|\tilde{u} \|_2^2   \|\tilde{u}_y \|_2^2) \|\om_x \|_2^2  +   C\|\tilde v\|_2^2  \| \tilde u_x \|_2^2  \|\nabla \om \|_2^2 \notag  \\
&\hspace{0.2 in} +  C  (\|\tilde{u} \|_2^2 +  \| \tilde{u}_x \|_2^2+    \| \tilde u_y \|_2^2  +   \|\tilde{u} \|_2^4  +  \|\tilde u\|_2^2\| \tilde{u}_y \|_2^2    )  \|\om_y \|_2^2   \notag\\
&\hspace{0.2 in} +    C( \|\omega\|_2^4  +  \|\omega\|_2^2) \|\om_x \|_2^2    + C( \|\text{curl} \, \mathbf f\|_2^2 +   \|\partial_{yy} f_1\|_2^2 )       \notag\\
&\leq   C  (\| \nabla \tilde{u} \|_2^2  +   \|\tilde{u} \|_2^2       +  \|\tilde{u} \|_2^4  +  \|\tilde{\mathbf u}\|_2^2\| \nabla \tilde{u} \|_2^2)  \|\nabla \om \|_2^2 +    C( \|\omega\|_2^4  +  \|\omega\|_2^2) \|\om_x \|_2^2 \notag\\
&\hspace{0.2 in}  + C( \|\text{curl} \, \mathbf f\|_2^2 +   \|\partial_{yy} f_1\|_2^2 ).
\end{align*}

By Gronwall's  inequality, for all $t\geq 0$,
\begin{align}      \label{v-9}
 &\|\nabla \om(t) \|_2^2    \notag\\
 & \leq  e^{C \int_0^t (\| \nabla \tilde{u}\|_2^2      +   \|\tilde{u} \|_2^2        +  \|\tilde{u} \|_2^4  +  \|\tilde{\mathbf u}\|_2^2\| \nabla \tilde{u} \|_2^2  ) ds}  \notag\\
&\hspace{0.3 in} \times  \left(   \|\nabla \om_0 \|_2^2 +     C \int_0^t  \left[( \|\omega\|_2^4  +  \|\omega\|_2^2) \|\om_x \|_2^2 +  \|\text{curl} \, \mathbf f\|_2^2    +   \|\partial_{yy} f_1\|_2^2         \right] ds    \right) 
    \notag\\
 &\leq        e^{C \left[ 1+ \sup_{0\leq s\leq t} \|\tilde{\mathbf u} \|_2^2     \right] \int_0^t \left(\|\tilde{u} \|_2^2  + \|\nabla \tilde{u} \|_2^2 \right) ds}       \notag\\
&\hspace{0.3 in} \times  \Big(   \|\nabla \om_0 \|_2^2 +C   \sup_{0\leq s\leq t}     ( \|\omega\|_2^4  +   \|\omega\|_2^2   )\int_0^t   \|\omega_x\|_2^2 ds       +  C\int_0^t \left(\|\text{curl} \, \mathbf f\|_2^2       +   \|\partial_{yy} f_1\|_2^2     \right) ds      \Big)   \notag \\ 
&\leq M\Big( \|\omega_0\|_2, \|\nabla \omega_0\|_2,    \int_0^t  \Big(\|\text{curl} \, \mathbf f\|_2 +\|\text{curl} \,\mathbf f\|_2^2    +     \|\partial_{yy} f_1\|_2^2        \Big) ds    \Big),
\end{align}
where we have used uniform bounds (\ref{v-2}), (\ref{v-2-0}), (\ref{v-2-3}) and (\ref{v-2-4}).

In particular,  if $\int_0^{\infty} \Big(\|\text{curl} \, \mathbf f\|_2 +\|\text{curl} \, \mathbf f\|_2^2   +     \|\partial_{yy} f_1\|_2^2     \Big) dt \leq C_0$, then $\|\nabla \om(t) \|_2^2$ is uniformly bounded by a constant for all time: 
\begin{align}  \label{v-10}
\|\nabla \om(t) \|_2^2 \leq M (\|\omega_0\|_2, \|\nabla \omega_0\|_2, C_0).
\end{align}
Due to uniform bounds (\ref{e-1}), (\ref{v-2}) and (\ref{v-10}), we obtain the uniform bound (\ref{H2b}) for $\|\mathbf u\|_{H^2}$.

\bigskip

\section{Stability of solutions around shear flows $(ay, 0)$}
This section is devoted to proving Proposition \ref{stability}, concerning the stability of solutions around shear flows $(ay, 0)$.

\begin{proof}
Assume $\text{curl} \, \mathbf f =0$. Let $(u+ay, v)$ be a solution of equation (\ref{PDE}). Then the vorticity $\omega = v_x - (u+ay)_y = v_x - u_y - a$. 
We denote by 
\begin{align} \label{om-0}
\om^*= v_x-u_y= \om +a. 
\end{align}
Due to the vorticity formulation (\ref{vor}), we have
\begin{equation}\label{om-1}
 \om^*_t+(u+a y) \om^*_x+ v\om^*_y  = \om^*_{xx} .
 \end{equation}
 We split $\om^* = \widetilde{\om^*} + \overline{\om^*}$. Then, 
 equation (\ref{om-1}) can be written as
 \begin{align}    \label{om-2}       
(\widetilde{\om^*})_t+     (\overline{\om^*})_t     +    (u+a y) (\widetilde{\om^*})_x+ v (\widetilde{\om^*})_y   + v (\overline{\om^*})_y     = (\widetilde{\om^*})_{xx} .
\end{align}
 
We take the inner product of equation (\ref{om-2}) with $\widetilde{\om^*}$. Notice that $\int_{\Omega}  (\overline{\om^*})_t   (\widetilde{\om^*}) dx dy 
= \int_0^1  (\overline{\om^*})_t   \overline{(\widetilde{\om^*})} dy =0$. Then, using integration by parts and divergence free condition, we have
\begin{align}     \label{om-3}
&  \frac{1}{2} \frac{d}{dt} \|\widetilde{\om^*}\|_2^2    + \|(\widetilde{\om^*})_x\|_2^2   = - \int_{\Omega} v (\overline{\om^*})_y (\widetilde{\om^*}) dx dy
\leq \|v\|_{\infty}  \|(\overline{\om^*})_y \|_2 \|\widetilde{\om^*} \|_2     \notag\\
&\leq \epsilon \|v\|_{\infty}^2 +  C \|(\overline{\om^*})_y \|_2^2 \|\widetilde{\om^*} \|_2^2.
\end{align}
Thanks to inequalities (\ref{ineq2}) and (\ref{Poin}), we have
\begin{align}     \label{om-4}
&\|v\|_{\infty}^2 \leq C ( \|v_y\|_2^2 +  \|v_x\|_2^2  +  \|v_{xy}\|_2^2) = C ( \|u_x\|_2^2 +  \|v_x\|_2^2  +  \|v_{xy}\|_2^2)  \notag\\
& =       C ( \|\tilde u_x\|_2^2 +  \|\tilde v_x\|_2^2  +  \|\tilde v_{xy}\|_2^2) 
\leq C ( \|\tilde u_{xx}\|_2^2 +  \|\tilde v_{xx}\|_2^2  +  \|\tilde v_{xy}\|_2^2).
\end{align}

Since $u_x + v_y =0$, we have $\tilde u_{xx} + \tilde v_{xy} =0$. Then, using (\ref{om-0}) and integration by parts, we have
\begin{align} \label{om-5}
\|(\widetilde{\om^*})_x\|_2^2 =  \int_{\Omega}  (\tilde v_{xx} - \tilde u_{xy})^2 dx dy  =
  \|\tilde v_{xx}\|_2^2 +  \|\tilde u_{xy}\|_2^2 +   \|\tilde v_{xy}\|_2^2 + \|\tilde u_{xx}\|_2^2.
 \end{align}

Combining (\ref{om-4}) and (\ref{om-5}) yields
\begin{align}  \label{om-6}
\|v\|_{\infty}^2 \leq C \|(\widetilde{\om^*})_x\|_2^2.
\end{align}

Therefore, by (\ref{om-3}) and (\ref{om-6}) and taking $\epsilon$ sufficiently small, we have
\begin{align}     \label{om-7}
\frac{d}{dt} \|\widetilde{\om^*}\|_2^2 + \|(\widetilde{\om^*})_x\|_2^2     \leq  C \|(\overline{\om^*})_y \|_2^2 \|\widetilde{\om^*} \|_2^2
\leq C \|(\overline{\om^*})_y \|_2^2 \|(\widetilde{\om^*} )_x\|_2^2.
\end{align}

By (\ref{om-0}), we see that $(\overline{\omega^*})_y =   (\bar{\om}+a)_y  = (\bar{\omega})_y$. 
Also, $\|(\bar{\omega})_y\|_2^2 = \int_{\Omega} (\int_0^1 \omega_y dx)^2 dx dy  \leq  \int_{\Omega} (\int_0^1 \omega_y^2 dx) dx dy =  \|\omega_y\|_2^2$. Hence,
\begin{align}     \label{om-8}
\|(\overline{\omega^*})_y\|_2^2 = \|(\bar{\omega})_y\|_2^2 \leq  \|\omega_y\|_2^2 \leq \|\nabla \omega\|_2^2 \leq M( \|\omega_0\|_2, \|\nabla \omega_0\|_2),
\end{align}
where the last inequality is due to (\ref{v-9}). 

Combining (\ref{om-7}) and (\ref{om-8}) yields
\begin{align*}
\frac{d}{dt} \|\widetilde{\om^*}\|_2^2  + \|(\widetilde{\om^*})_x\|_2^2     \leq    C \cdot     M( \|\omega_0\|_2, \|\nabla \omega_0\|_2)    \|(\widetilde{\om^*} )_x\|_2^2.
\end{align*}

Therefore, by letting $\|\omega_0\|_2$ and $\|\nabla \omega_0\|_2$ sufficiently small so that
\begin{align}    \label{om-9}
C \cdot     M( \|\omega_0\|_2, \|\nabla \omega_0\|_2) \leq \frac{1}{2}\,,
\end{align}
where the constant $M( \|\omega_0\|_2, \|\nabla \omega_0\|_2)$ comes from (\ref{v-9}), we have
\begin{align*}
\frac{d}{dt} \|\widetilde{\om^*}\|_2^2  + \frac{1}{2}\|(\widetilde{\om^*})_x\|_2^2 \leq 0.
\end{align*}
It follows that 
\begin{align*}
\frac{d}{dt} \|\widetilde{\om^*}\|_2^2  + \alpha\|\widetilde{\om^*}\|_2^2 \leq 0,
\end{align*}
for some constant $\alpha>0$.

By Gronwall's  inequality, we have
\begin{align*}
\|\widetilde{\om^*}(t)\|_2^2  \leq
 e^{- \alpha t}    \|\widetilde{\om^*} (0)\|_2^2,
\end{align*}
for all $t\geq 0$.
Therefore, the oscillation part decays exponentially to zero, if $\|\omega_0\|_2$ and $\|\nabla \omega_0\|_2$ are sufficiently small.
\end{proof}

\bigskip

\section{Asymptotic solutions}
This section is devoted to proving Proposition \ref{INFTY}. It is about the asymptotic behavior of solutions if the external force vanishes. 
First, we need a technical lemma.

\begin{lemma} \label{lemmA}
Let $f(t)$ and $g(t)$ be $C^1$ functions on $\mathbb R^+$ such that $f(t) \geq 0$ for all $t>0$. Assume
$\frac{df}{dt} + \frac{dg}{dt} +\lambda f \leq 0$ for $t>0$ where $\lambda >0$, and $\lim_{t\rightarrow \infty} g(t)=A$ for some constant $A$. Then, 
$\lim_{t\rightarrow \infty}  f(t) =0$.
\end{lemma}

\begin{proof}
Multiply both sides of $\frac{df}{dt} + \frac{dg}{dt} +\lambda f \leq 0$ by $e^{\lambda t}$, then 
\begin{align} \label{A-0}
 \frac{d}{dt}(e^{\la t}f)+ e^{\la t} \frac{dg}{dt}  \leq 0.
\end{align}
Integrate (\ref{A-0}) over $[t_0,t]$, for $0<t_0 < t$, to get
\[
e^{\la t}f (t) - e^{\la t_0}f (t_0) + \int_{t_0}^t \Big(e^{\la s} \frac{dg}{dt}\Big) ds  \leq 0.
\]
Integration by parts leads to
\[
e^{\la t}f (t) -e^{\la t_0}f (t_0)+ e^{\la t}g (t) - e^{\la t_0}g (t_0)- \la \int_{t_0}^t e^{\la s} g(s) ds  \leq 0.
\]
Thus,
\begin{align} \label{A-1}
f (t) - e^{-\la (t-t_0)} f(t_0) + g (t) - e^{-\la (t-t_0)} g(t_0) - \la e^{-\la t} \int_{t_0}^t e^{\la s} g(s) ds  \leq 0.
\end{align}
Since $\lim_{t\rightarrow \infty}  g(t)= A$, by L'Hôpital's rule, we have
\[
\lim_{t\rightarrow \infty}   \la e^{-\la t} \int_{t_0}^t e^{\la s} g(s) ds  = \lim_{t\rightarrow \infty}  g(t) =A.
\]
Then, by letting $t\rightarrow \infty$ on (\ref{A-1}), we get
$\lim_{t\rightarrow \infty}   f (t) \leq 0$. Since we assume $f(t) \geq 0$ for all $t>0$, 
we conclude that $\lim_{t\rightarrow \infty}   f (t) = 0$.

\end{proof}

Now we are ready to prove Proposition \ref{INFTY}.

\begin{proof}
Assume $\mathbf u_0 \in \mathcal H \cap (H^1(\Omega))^2$ and $\mathbf f=0$. Let $\mathbf u =(u,v)$ be the global solution of equation (\ref{PDE}). 
Since $u_x + v_y =0$, we get $\bar v_y =0$. Then, because of the boundary condition $v|_{y=0} = v|_{y=1}=0$, we see that $\bar v (y)=0$ for all $y\in [0,1]$. 
Hence, $\overline{v \bar u_y} = \bar v \bar u_y =0$. Then, taking the horizontal mean on equation (\ref{be-3}) gives
\begin{align} \label{la-1}
\bar u_t + \overline{v u_y} = \bar u_t + \overline{v \tilde u_y} =0.
\end{align}

We estimate
\begin{align} \label{la-2}
&\|\overline{v \tilde u_y} \|_{L^2(0,1)}= \Big(\int_0^1 \Big[\int_0^1 (v \tilde u_y) dx  \Big]^2 dy \Big)^{1/2}
\leq \|\tilde u_y\|_2  \Big(\sup_{y\in [0,1]} \int_0^1 v^2 dx \Big)^{1/2} \notag\\
&\leq  C\|\tilde u_y\|_2 \|v_y\|_2 \leq C\|\tilde u_{xy}\|_2 \|u_x\|_2=  C\|u_{xy}\|_2 \|u_x\|_2 \leq C (\|u_{xy}\|_2^2 +  \|u_x\|_2^2).
\end{align}

Integrating (\ref{la-1}) over $[s,t]$ gives that $\bar u(t) - \bar u(s) = -\int_s^t  \overline{v \tilde u_y} d\alpha$, and along with (\ref{la-2}), we obtain
\begin{align} \label{la-3}
\|\bar u(t) - \bar u(s)\|_{L^2(0,1)} \leq \int_s^t   \|\overline{v \tilde u_y}\|_{L^2(0,1)} d\alpha  \leq C \int_s^t  (\|u_{xy}\|_2^2 + \|u_x\|_2^2) d\alpha,
\end{align}
for any $t>s\geq 0$. 

Because of global estimates (\ref{e-2}) and (\ref{v-2-0}), we have 
\begin{align} \label{la-4}
\int_0^{\infty}  (\|u_{xy}\|_2^2 + \|u_x\|_2^2) d\alpha \leq \|\omega_0\|_2^2 + \|\mathbf u_0\|_2^2,
\end{align}
since $\mathbf f =0$.

Due to (\ref{la-3}) and (\ref{la-4}), we obtain for any $\epsilon>0$, there exists $T>0$ such that 
$\|\bar u(t) - \bar u(s)\|_{L^2(0,1)} < \epsilon$ when $t, s\geq T$. Thus,
there exists $\phi(y) \in L^2(0,1)$ such that 
\begin{align} \label{la-5}
\lim_{t\rightarrow \infty} \bar u =  \phi \;\; \text{in} \;\; L^2(0,1).
\end{align}
Since $\phi(y)$ depends solely on $y$, the velocity field $(\phi(y),0)$ is a steady state solution of equation (\ref{PDE}). 

Finally, due to the energy identity $\frac{1}{2} \frac{d}{dt} (\|u\|_2^2 + \|v\|_2^2) + \|u_x\|_2^2 + \|v_x\|_2^2 =0$ and the fact that $\bar v =0$, we obtain
\begin{align*}
\frac{1}{2} \frac{d}{dt} (\|\tilde u\|_2^2 +  \|\bar u\|_2^2   +   \|\tilde v\|_2^2) + \|\tilde u_x\|_2^2 + \|\tilde v_x\|_2^2 =0,  \;\; \text{for} \; t>0.
\end{align*}
It follows that 
\begin{align*}
 \frac{d}{dt} (\|\tilde u\|_2^2 +  \|\tilde v\|_2^2) +   \frac{d}{dt} (\|\bar u\|_2^2)   + \lambda (\|\tilde u\|_2^2 + \|\tilde v\|_2^2) \leq 0,      \;\; \text{for} \; t>0.
 \end{align*}
From (\ref{la-5}), we have $\lim_{t\rightarrow \infty} \|\bar u\|_2 = \|\phi\|_2$, then using Lemma \ref{lemmA}, we get
\begin{align} \label{la-6}
\lim_{t\rightarrow \infty} (\|\tilde u\|_2^2 +  \|\tilde v\|_2^2)  =0.
\end{align}
Because of (\ref{la-5}) and (\ref{la-6}) as well as $\bar v=0$, we conclude 
\begin{align*}
\lim_{t\rightarrow \infty}(u,v) = (\phi, 0)    \;\;   \text{in}  \;\; \mathcal H,
\end{align*}
where $(\phi(y),0)$ is a steady state solution of equation (\ref{PDE}) when $\mathbf f=0$.
\end{proof}

\end{document}